\documentclass[a4paper,reqno,12pt]{amsart} 
\usepackage[latin1]{inputenc}
\usepackage[T1]{fontenc}
\usepackage[english]{babel}
\usepackage{appendix}
\usepackage[english]{minitoc}
\usepackage[nice]{nicefrac}
\usepackage[top=2.5cm, bottom=2.8cm, left=2.5cm, right=2.5cm]{geometry}
\numberwithin{equation}{section} \makeatletter\@addtoreset{equation}{section}
\usepackage{amsmath,amssymb,amsbsy,latexsym, amsfonts,graphics,ulem,hhline,dsfont,mathrsfs,color}
\usepackage{fancyhdr,rotating,fancybox,colortbl,pifont,setspace,enumerate,multicol,varioref,lmodern,textcomp,euscript}
\usepackage[pdftex]{hyperref}

\newcommand{\PbP}{\frac{\partial^2}{\partial z\partial \bz}}

\newcommand{\C}{\mathbb C}      \newcommand{\R}{\mathbb R}
\newcommand{\Z}{\mathbb Z} 
\newcommand{\bz}{\overline{z}}

\newcommand{\bzeta}{\overline{\zeta}}
\newcommand{\pz}{\partial_{z}}
\newcommand{\pzk}{\partial_{z}^k}

\newcommand{\pzm}{\partial_{z}^m}
\newcommand{\pzn}{\partial_{z}^n}

\newcommand{\pbz}{\partial_{\bz}}
\newcommand{\pzbz}{\partial_{z}\partial_{\bz}}

\newcommand{\norm}[1]{\left\Vert#1\right\Vert}  
  \newcommand{\scal}[1]{\left<#1\right>}  
\newtheorem {theorem}{Theorem}[section]            \newtheorem {lemma}[theorem]{Lemma}
   \newtheorem {corollary}[theorem]{Corollary}     \newtheorem {remark}[theorem]{Remark}
\newtheorem {proposition}[theorem]{Proposition}       
         
\begin{document}

\title{On a novel class of polyanalytic Hermite polynomials}
\author[A. Benahmadi]{Abdelhadi Benahmadi}
\author[A. Ghanmi]{Allal Ghanmi}

 \address{ A.G.S. - L.A.M.A., CeReMar,
          Department of Mathematics, P.O. Box 1014,  Faculty of Sciences,
          Mohammed V University in Rabat, Morocco}
\email{abdelhadi.benahmadi@gmail.com}
\email{allal.ghanmi@um5.ac.ma}

\begin{abstract}
We carry out some algebraic and analytic properties of a new class of orthogonal polyanalytic polynomials, including their operational formulas, recurrence relations, generating functions, integral representations and different orthogonality identities.
We establish their connection and rule in describing the $L^2$--spectral theory of some special second order differential operators of Laplacian type acting on the $L^2$--gaussian Hilbert space on the whole complex plane.
We will also show their importance in the theory of the so-called rank--one automorphic functions on the complex plane. In fact, a variant subclass leads to an orthogonal basis of the corresponding $L^2$--gaussian Hilbert space on the strip.
\end{abstract}

\maketitle



\section{ \quad Introduction}

 The classical real Hermite polynomials (see e.g. \cite{Hermite1864-1908,Rainville71})
 are extensively studied in the mathematics literature and they have found interesting applications in various branches of mathematics, physics and technology, see for examples \cite{Rainville71,Szego75,Thangavelu93} and the references therein.
 Natural extensions to the two real variables can be obtained by considering the tensor product $ H_{m,n}(x,y) = H_{m}(x) H_{n}(y) $ or also by replacing the real $x$ in $H_{m}(x)$ by the complex variable $z$ giving rise to the holomorphic Hermite polynomials  (see e.g.  \cite{Szego75,Ismail13a}).
  \begin{align}\label{HolHm}
  H_n(z) = (-1)^{n}e^{z^2}\frac {d^n}{dz^n }(e^{-z^2}).
  \end{align}
 This last class has been introduced in the study of some analytic function spaces \cite{Eijndhoven-Meyers1990,Karp2001,Chihara2017} and showed to be useful for the coherent states theory \cite{CoftasGazeau10,GazeauSzafraniec2011}. Their combinatoric has been studied in \cite{Ismail13a}.
 Another generalization is given by the univariate complex Hermite polynomials $H_{m,n}(z,\bz )$
introduced by Ito in the context of complex Markov process \cite{Ito52}
as well as their generalized version $G_{m,n}(z,\bz )$ considered by the second author in \cite{Gh08JMAA}.
Both are special examples of polyanalytic polynomials of one complex variable for satisfying the generalized Cauchy equation $\dfrac{\partial ^k}{\partial \bz^{n+1}}=0$.
The curious reader can refer to \cite{Shigekawa87,Matsumoto96,IntInt06,Gh08JMAA,Gh13ITSF,IsmailTrans2016, Gh2017} and the references therein for their basic properties and their applications.

The following general class of polyanalytic polynomials including the $H_{m,n}(z,\bz )$ and the $G_{m,n}(z,\bz )$ is suggested in \cite{AiadGh10JMAA}
  \begin{align}\label{chp}
         G_{m,n}^h(z,\bz )=(-1)^{m+n}e^{\nu |z|^2 - h(z)}\dfrac{\partial ^{m+n}}{\partial \bz^{m} \partial z^{n}} \left(e^{-\nu |z|^2 + h(z)}\right) ,
         \end{align}
where $h(z)$ is a given entire function. They appear naturally, when dealing with the spectral theory of a special magnetic Laplacian leaving invariant the space of mixed automorphic functions \cite{AiadGh10JMAA}.

In the present paper, we consider the particular case $G_{0,n}^{h_0^{\alpha,\xi}}(z,\bz )$ corresponding to the special holomorphic function $h_0^{\alpha,\xi}(z)= \alpha z^2 + \xi z$ for arbitrary real $\alpha$ and complex number $\xi$. In fact, we consider
\begin{align}\label{RodriguesIm}
I_n^{\nu,\alpha}\left(z,\bz|\xi\right) = (-1)^n e^{\nu   z \bz -\alpha  z^2 - \xi z}  \dfrac{\partial ^{n}}{\partial z^{n}}  \left(e^{- \nu   z \bz + \alpha  z^2 + \xi z} \right),
\end{align}
Such class of functions can be seen as the polyanalytic generalization of the holomorphic $H_n(z)=  I_n^{0,-1}\left(z,\bz|0\right)$ as well as the monomials  $I_n^{1,0}\left(z,\bz|0\right)= \bz^n$.
 The consideration of this class is motivated by their importance in the theory of the automorphic functions on the complex plane with respect to given rank--one discrete subgroup $\Gamma=\Z$ of $(\C,+)$ (see Section 7). In fact, the specific case of $\xi = 2i\pi (\beta+k)$, with $\beta \in \R$ and $k\in \Z$, leads to
\begin{align}\label{GhImnPoly-1}
 I_{n,k}^{\nu,\alpha,\beta}\left(z,\bz|\xi\right) & :=  I_n^{\nu,\alpha}\left(z,\bz|2i\pi (\beta+k) \right)
\end{align}
which for fixed nonnegative integer $n$ gives rise to an orthogonal basis of the $n^{th}$ $L^2$--eigenspace of a Schr\"odinger operator acting on some
$L^2$--sections over the strip $\C/\Z$ of the $L^2$--line bundle $L=(\C\times\C)/\Z$, constructed as the quotient of the trivial bundle over $\C$ by considering the $\Z$--action \cite{GhIn2013,Souid2015}.

 Our main purpose here is to explore some basic properties of $I_{n}^{\nu,\alpha}\left(z,\bz|\xi\right)$
  such as operational representations, the recurrence relations, differential equations they satisfy,
 orthogonality relations,  Rodrigues' formula and quadratic formula of Nielsen type as well as the explicit formula in terms of the Hermite polynomials.
 We also provide generating functions and integral representations, including the one involving a Fourier--Wigner transform with a special window function closed to the classical Mehler's kernel.
	In course of our investigation, we present two interesting applications.
	The first one is related to the concrete description of the spectral theory of some specific second order differential operator of Laplacian type acting on the free Hilbert space $L^2(\C; e^{-\nu|z|^2}dxdy)$. The second application involves the subclass $ I_{m,n}^{\alpha,\beta}\left(z,\bz|\xi\right)$ in \eqref{GhImnPoly-1} and reproves the fact that they form a complete orthogonal system of the space $L^2(\C/\Z; e^{-\nu|z|^2}dxdy)$ of $L^2$--rank--one automorphic functions.

The layout of the paper is as follows. In Section 2, we introduce and study some basic properties of $ I_{n}^{\nu,\alpha}\left(z,\bz|\xi\right)$. In Section 3, we provide four kinds of generating functions. In Section 4, we discuss the orthogonality properties. Section 5 is concerned with the problem of providing integral representations of $ I_{n}^{\nu,\alpha}\left(z,\bz|\xi\right)$. The polyanaliticity of these polynomials and the differential equations they satisfy are  discussed in Section 6.
While in Section 7, we discuss their importance in studying the spectral properties of rank--one automorphic theta functions. Some concluding remarks close the paper.

\section{ Preliminary results}

This section incorporates a preliminary study of the polynomials $I_n^{\nu,\alpha}\left(z,\bz|\xi\right)$ (abbreviated  sometimes as $I_n^{\nu,\alpha}$).
For the unity of the formulation, we put $I_n^{\nu,\alpha} = 0 $ whenever $n<0$.
Notice for instance that
 \begin{align}
 I_1^{\nu,\alpha}\left(z,\bz|\xi\right) & =  \nu  \bz - 2 \alpha  z - \xi . \label{I1}
 \end{align}
 The first result in this section concerns useful operational formulas for $I_n^{\alpha}\left(z,\bz|\xi\right)$.

\begin{proposition}\label{prop-Repr2}
The polynomials $ I_n^{\nu,\alpha}\left(z,\bz|\xi\right)$ can be realized as
\begin{align}
 I_n^{\nu,\alpha}\left(z,\bz|\xi\right)  & =  e^{-\alpha  z^2 - \xi z}  \left(-\pz + \nu  \bz\right)^n (e^{\alpha  z^2 + \xi z}) \label{GhInPoly-1}
  \\& =e^{-\alpha z^2}(-\partial_z+\nu \overline{z}-\xi)^n e^{\alpha z^2}. \label{GhInPoly-11}
\end{align}
Moreover, the first order differential operators $\pz - I_1^{\nu,\alpha}$ and $\pbz$ are respectively the corresponding arising and lowering operators in the sense that we have
      \begin{align}
      \pz I_n^{\nu,\alpha} &= I_1^{\nu,\alpha}I_n^{\nu,\alpha} - I_{n+1}^{\nu,\alpha} \label{DerIm}
            \end{align}
and
      \begin{align}
      \pbz I_n^{\nu,\alpha} &= \nu n I_{n-1}^{\nu,\alpha}. \label{DerbarIm}
      \end{align}
\end{proposition}

\begin{proof}
The representation \eqref{GhInPoly-1} as well as \eqref{GhInPoly-11} follow from the Rodrigues' type formula \eqref{RodriguesIm} making use of the fact
\begin{align*}
\left(\pz - \nu  \bz + \xi \right)^{n}(f) &= e^{-\xi z} \left(\pz - \nu  \bz \right)^{n}(e^{\xi z} f)
=   e^{ \nu  z \bz -\xi z } \pzn  \left(e^{- \nu z \bz+\xi z } f\right)
\end{align*}
for sufficiently differentiable function $f$.
Both \eqref{RodriguesIm} and \eqref{GhInPoly-1} can be used to establish \eqref{DerIm}. The proof we provide below is based on \eqref{GhInPoly-1}. Indeed, direct computation yields
 \begin{align}
     \pz I_n^{\nu,\alpha}  & = -  \left( 2 \alpha  z + \xi \right)  I_n^{\nu,\alpha}    + \left(-1\right)^n e^{-\alpha z^2-\xi z} \pz \left(\pz - \nu  \bz\right)^n e^{\alpha z^2+\xi z} \label{ProofRecF}.
 \end{align}
Hence, by rewriting the $z$--derivation operator in the form $\pz = \left(\pz - \nu \bz\right) + \nu  \bz$, keeping in mind the expression of $ I_1^{\nu,\alpha}$ given through \eqref{I1},  we obtain
 \begin{align*}
     \pz I_n^{\nu,\alpha}
            & =   \left( \nu  \bz - 2 \alpha  z -  \xi \right)  I_n^{\nu,\alpha}    - I_{n+1}^{\nu,\alpha}
= I_1^{\nu,\alpha}I_n^{\nu,\alpha}  - I_{n+1}^{\nu,\alpha} .
 \end{align*}
This proves \eqref{DerIm}.  To establish \eqref{DerbarIm}, we make use of
\begin{align} \label{Dbar-e1}
\pbz \left(\pz - \nu   \bz\right)^n h
 =  - \nu   n\left(\pz - \nu   \bz\right)^{n-1} h,
\end{align}
which holds true for any holomorphic function $h$ and in particular for $h(z)=e^{\alpha z^2 + \xi  z} $.
Therefore, we obtain
 \begin{align*}
 \pbz I_n^{\nu,\alpha}     &= (-1)^n e^{-\alpha z^2-\xi z} \pbz \left[  \left(\pz - \nu   \bz\right)^n e^{\alpha z^2+\xi z}\right]
                         \stackrel{\eqref{Dbar-e1}}{=}
                        \nu  n I_{n-1}^{\nu,\alpha}.
 \end{align*}
\end{proof}

The following result gives another interesting representation of the polynomials $I_{n}^{\nu,\alpha}$.

\begin{proposition}\label{prop-Rep3}
The polynomials $I_n^{\nu,\alpha}$ can be represented as
     \begin{align}\label{ImRep3}
      I_n^{\nu,\alpha}= \left(-\pz + I_1^{\nu,\alpha}\right)^n \cdot (1) .
      \end{align}
Subsequently,
we have
\begin{align}\label{DerImExp}
\pz  I_n^{\nu,\alpha}= - 2\alpha n   I_{n-1}^{\nu,\alpha}.
      \end{align}
\end{proposition}

\begin{proof}
Notice first that \eqref{DerIm} can be rewritten as
$ \left(-\pz +  I_1^{\nu,\alpha}\right) I_k^{\nu,\alpha}=  I_{k+1}^{\nu,\alpha}$.
Therefore, we get
     \begin{align}\label{Im3k}
      \left(-\pz + I_1^{\nu,\alpha}\right)^n I_k^{\nu,\alpha}=  I_{k+n}^{\nu,\alpha}
      \end{align}
for any arbitrary nonnegative integers $n$ and $k$. Hence, for $k=0$, we obtain
   $ \left(-\pz + I_1^{\nu,\alpha}\right)^n \cdot (1) =  I_{n}^{\nu,\alpha}$.
This proves \eqref{ImRep3}.
The proof of \eqref{DerImExp} lies essentially in the fact that
\begin{align}\label{DerIun}
\pz \left(-\pz + I_1^{\nu,\alpha}\right)^{n} \cdot (1) = - 2\alpha n    \left(-\pz + I_1^{\nu,\alpha}\right)^{n-1} \cdot (1).
   \end{align}
\end{proof}

\begin{remark}\label{remnu2alp}
By comparing \eqref{DerbarIm} and \eqref{DerImExp}, we get
	$ \pz  I_n^{\nu,\alpha}= - \frac{2\alpha}{\nu} \pbz  I_{n}^{\nu,\alpha}$.
	Thus for $\alpha >0$ and $\nu=2\alpha$, the corresponding polynomials
	$I_n^{\nu,\nu/2}$ are antisymmetric in the sense that $I_{n}^{\nu,\nu/2}(z,\bz)= - I_{n}^{\nu,\nu/2}(\bz,z)$.
	Moreover, they depends only in the imaginary part of $z$.
\end{remark}

Combination of \eqref{DerIm} and \eqref{DerImExp} yields the following

\begin{corollary}\label{cor-RecF11}
The polynomials $I_n^{\nu,\alpha}$ satisfy the three term recurrence formula
\begin{align}\label{RecF11}
   I_{n+1}^{\nu,\alpha}= I_1^{\nu,\alpha}I_n^{\nu,\alpha} +   2\alpha n   I_{n-1}^{\nu,\alpha}.
      \end{align}
\end{corollary}

\begin{remark}
 In the proof of Proposition \ref{prop-Repr2} (resp. Proposition \ref{prop-Rep3}),
  we have make use of the identity \eqref{Dbar-e1} (resp. \eqref{DerIun}).
 These identities can be handled by induction. They also are particular cases of the well--established algebraic identity
 $A B^{n+1} - B^{n+1} A = \lambda n B^n$ whenever $AB-BA = \lambda Id$.
\end{remark}

It may be of interest to point out that $ I_n^{\nu,\alpha}$ are also polynomials in $\xi$ with degree $n$.
This can be seen easily in virtue of the following

\begin{lemma} \label{lem-Der-zxi} We have
 \begin{align}
        \label{Der-zxi}
       2 \alpha  \partial_\xi I_n^{\nu,\alpha} = \pz I_n^{\nu,\alpha} = - 2\alpha n I_{n-1}^{\nu,\alpha}
       \end{align}
 and consequently, the following recurrence formula
\begin{align}\label{RecF2}
      I_{n+1}^{\nu,\alpha} = I_1^{\nu,\alpha}I_n^{\nu,\alpha} - 2 \alpha  \partial_\xi I_n^{\nu,\alpha}
 \end{align}
holds true.
\end{lemma}

 \begin{proof}
Direct computation, starting from \eqref{ProofRecF} and using the fact that $\pz$ and $\pz - \nu \bz $ commute, entails
\begin{align*}
     \pz I_n^{\nu,\alpha}
& = -  (2 \alpha  z + \xi )  I_n^{\nu,\alpha} + (-1)^n e^{-\alpha z^2-\xi z}  \left(\pz - \nu \bz\right)^n \left(\pz e^{\alpha z^2+\xi z}\right) \\
& =    2 \alpha  \left\{ - z  I_n^{\nu,\alpha} + (-1)^n e^{-\alpha z^2-\xi z}  \left(\pz - \nu \bz\right)^n \left(  z  e^{\alpha z^2+\xi z} \right)\right\} \\
& =    2 \alpha (-1)^n  \partial_\xi\left\{ e^{-\alpha z^2-\xi z} \left(\pz - \nu \bz\right)^n \left(  z  e^{\alpha z^2+\xi z}\right) \right\} \\
& =    2 \alpha \partial_\xi I_n^{\nu,\alpha}.
                                \end{align*}
 Insertion of \eqref{Der-zxi} in \eqref{DerIm} yields \eqref{RecF2}.
\end{proof}

\begin{remark}
The property \eqref{DerIm} in Proposition \ref{prop-Repr2} (resp. \eqref{DerbarIm} in Proposition \ref{prop-Repr2} and \eqref{Der-zxi} in Lemma \ref{lem-Der-zxi}) shows that the considered polynomials $I_n^{\nu,\alpha}$ constitute an Appell sequence with respect to $z$ (resp. $\bz$ and $\xi$).
\end{remark}

Added to the Rodrigues' formula \eqref{RodriguesIm} defining $I_n^{\nu,\alpha}$, these polynomials admit a second useful Rodrigues' formula.

\begin{theorem}\label{thm-RodF}
	We have
	\begin{align}\label{RodF}
	I_n^{\nu,\alpha}(z,\bz|\xi )&= (-1)^n e^{\frac{-\left(I_{1}^{\nu,\alpha}(z,\bz|\xi )\right)^2}{4\alpha} }\dfrac{\partial^n}{\partial z^n} \left( e^{\frac{\left(I_{1}^{\nu,\alpha}(z,\bz|\xi )\right)^2}{4\alpha} }\right)
	\end{align}
\end{theorem}

\begin{proof}
	We proceed by induction. Obviously, \eqref{RodF} holds good for $n=0$ and $n=1$.
	In fact $ \pz \left( e^{\frac{(I_{1}^{\nu,\alpha})^2}{4\alpha} }\right)= - I_{1}^{\nu,\alpha}e^{\frac{-(I_{1}^{\nu,\alpha})^2}{4\alpha} }$.
	Now, assume that  \eqref{RodF} holds true for every nonnegative integer $k \leq n$, for given fixed $n$. Since $ \pz^j \left(I_{1}^{\nu,\alpha} \right)=0$ for $j=2,3,\cdots$, we get 
	$$ \pz^{n+1} \left( e^{\frac{(I_{1}^{\nu,\alpha})^2}{4\alpha} }\right)
	=  - I_{1}^{\nu,\alpha}  \pz^{n} \left( e^{\frac{(I_{1}^{\nu,\alpha})^2}{4\alpha} }\right) + n \pz \left(-I_{1}^{\nu,\alpha} \right) \pz^{n-1} \left( e^{\frac{(I_{1}^{\nu,\alpha})^2}{4\alpha} }\right),$$
	so that
	\begin{align*} (-1)^{n+1}e^{\frac{-(I_{1}^{\nu,\alpha})^2}{4\alpha} }\pz^{n+1} \left( e^{\frac{(I_{1}^{\nu,\alpha})^2}{4\alpha} }\right)
	&=   I_{1}^{\nu,\alpha}  (-1)^{n} e^{\frac{-(I_{1}^{\nu,\alpha})^2}{4\alpha} }\pz^{n} \left( e^{\frac{(I_{1}^{\nu,\alpha})^2}{4\alpha}}\right)
	\\ & \quad + 2\alpha n (-1)^{n-1} e^{\frac{-(I_{1}^{\nu,\alpha})^2}{4\alpha} }\pz^{n-1} \left( e^{\frac{(I_{1}^{\nu,\alpha})^2}{4\alpha} }\right)
	\\&=
	I_{1}^{\nu,\alpha}  I_n^{\nu,\alpha}  + 2\alpha n I_{n-1}^{\nu,\alpha}  .
	\end{align*}
	Thus, one arrives at the desired result by means of the recurrence formula \eqref{RecF11}.
\end{proof}

The previous result shows in particular that the polynomials  $I_n^{\nu,\alpha} $ should be closely connected to the univariate Hermite polynomials $H_n(x)$. In fact, the following result asserts that they are essentially the $H_n$ in the variable $I_1^{\nu,\alpha}$.

\begin{corollary}\label{cor-expIm}
	The explicit expression of $I_n^{\nu,\alpha} $ in terms of the classical Hermite polynomials is given by
	\begin{align}  \label{CloHer}
	I_n^{\nu,\alpha}(z,\bz|\xi)  & = (-i)^n \alpha^{n/2} H_n\left(\frac{ i I_1^{\nu,\alpha}}{ 2 \alpha^{1/2}} \right)
	 = (-i)^n \alpha^{n/2} H_n\left(\frac{ 2\alpha z - \nu \bz + \xi}{ 2 i \alpha^{1/2}} \right) ,
	\end{align}
	with $\alpha\ne 0$ and the convention that $\alpha^{1/2}= i\sqrt{|\alpha|}$ if $\alpha<0$.
\end{corollary}

\begin{remark}
	For the particular case of $\nu=2\alpha>0$, the result of Corollary \ref{cor-expIm} shows that $ I_m^{\nu,\nu/2} (z,\bz|\xi)$ are polynomials in $\Im\left(z\right)$ and read simply
	\begin{align}  \label{CloHerPC}
	I_n^{\nu,\nu/2}(z,\bz|\xi)   = (-i)^n \left(\frac{\nu}{2}\right)^{n/2} H_n\left(\frac{ 2 \nu \Im(z) +\xi }{ i (2 \nu )^{1/2}} \right).
	\end{align}
	This is in accordance with Remark \ref{remnu2alp}
	The special case of adequate $\xi$ ($\xi=2i\pi(\beta+k)$) will be reconsidered in Section 7 when dealing with rank--one automorphic functions.
\end{remark}

The following result gives the expression of $I_n^{\nu,\alpha}(z,\bz|\xi)$ in terms of the tensor product $H^{\tau}_{j}(x) H^{\mu}_k(y)$ of the rescaled real Hermite polynomials,
  \begin{align*}
H^\tau_k(t) = (-1)^{n}e^{\tau t^2}\frac {d^n}{dt^n }(e^{-\tau t^2}).
\end{align*}

\begin{proposition}\label{propHH} For $\nu$ and $\alpha$ such that $2|\alpha|<\nu$, we have
	\begin{align}
	I_n^{\nu,\alpha}(z,\bz|\xi) =
	\frac{1}{2^n} \sum_{k=0}^n (-i)^k \binom{n}{k} H^{\nu-2\alpha}_{n-k}\left( x- \frac{\Re(\xi)}{\nu-2\alpha}\right)  H^{\nu+2\alpha}_k\left( y + \frac{\Im(\xi)}{\nu+2\alpha}\right) .
	\end{align}
\end{proposition}

\begin{proof}
Notice first that by considering the first order differential operators
$$A_x^{\nu,\alpha,\xi} f = -\frac{1}{2}\left( \partial_x	- 2 (\nu-2\alpha)x +2\Re(\xi) \right) f$$
and
$$B_y^{\nu,\alpha,\xi} f = -\frac{1}{2}\left( \partial_y	- 2 (\nu+2\alpha)y -2\Im(\xi) \right) f,$$
we clearly have $[A_x^{\nu,\alpha,\xi},B_y^{\nu,\alpha,\xi}]=0$. Moreover,
$$ (A_x^{\nu,\alpha,\xi})^n \cdot (1)  =
\frac{1}{2^n}  H^{\nu-2\alpha}_{n}\left( x- \frac{\Re(\xi)}{\nu-2\alpha}\right)  $$
and
$$
(B_y^{\nu,\alpha,\xi})^n \cdot (1)  =
\frac{1}{2^n}
H^{\nu+2\alpha}_n\left( y + \frac{\Im(\xi)}{\nu+2\alpha}\right)$$
which readily follows by induction from the fact
$$ \left( \partial_t - 2\tau t + \mu\right) f = e^{\tau\left( t-\frac{\mu}{2\tau}\right)^2} \partial_t\left( e^{-\tau\left( t-\frac{\mu}{2\tau}\right)^2} f \right) .$$
Now from Proposition \ref{prop-Rep3}, we have
\begin{align*}
I_n^{\nu,\alpha}&= \left(-\pz + I_1^{\nu,\alpha}\right)^n \cdot (1)
\\&=
\left( A_x^{\nu,\alpha,\xi} - i B_y^{\nu,\alpha,\xi} \right)^n \cdot (1) \\
&= \sum_{k=0}^n (-i)^k \binom{n}{k} \left( A_x^{\nu,\alpha,\xi}\right) ^{n-k} \cdot (1)  \left( B_y^{\nu,\alpha,\xi}\right) ^{k} \cdot (1)
 \\&=	\frac{1}{2^n} \sum_{k=0}^n (-i)^k \binom{n}{k} H^{\nu-2\alpha}_{n-k}\left( x- \frac{\Re(\xi)}{\nu-2\alpha}\right)  H^{\nu+2\alpha}_k\left( y + \frac{\Im(\xi)}{\nu+2\alpha}\right) .
 \end{align*}
\end{proof}

We conclude this section by proving a Nielsen's identity for these polynomials, which consists of expressing
 $I_n^{\nu,\alpha}$ as a weighted sum of a product of the same polynomials.
Namely, we have

\begin{theorem}\label{thm-Nielson}
The Nielsen's identity for the polynomials $I_n^{\nu,\alpha}$ reads
        \begin{align}  \label{NielsenIm}
               I_{m+n}^{\nu,\alpha}& = m! n! \sum\limits_{k=0}^{\min{(m,n)}} \frac {(2\alpha )^k}{k!}
                           \frac{ I_{m-k}^{\nu,\alpha}}{(m-k)!}\frac{ I_{n-k}^{\nu,\alpha}}{(n-k)!}.
        \end{align}
\end{theorem}

\begin{proof}
Starting from the Rodrigues' formula \eqref{RodriguesIm}, we can easily see that $I_{m+n}^{\nu,\alpha}$ takes the form
$$
I_{m+n}^{\nu,\alpha}= (-1)^{m}  e^{\nu   z \bz -\alpha  z^2 - \xi z}  \pzm  \left(  e^{-\nu   z \bz +\alpha  z^2 + \xi z}  I_n^{\nu,\alpha} \right)
.$$
Now, by means of the Leibniz formula combined with \eqref{RodriguesIm} and
\begin{align}\label{DerImExpg}
\pz^k I_m^{\nu,\alpha} (z,\bz|\xi)=  \frac{m!(-2\alpha)^k}{(m-k)!} I_{m-k}^{\nu,\alpha},
   \end{align}
   which follows by induction starting from \eqref{DerImExp}, we obtain
\begin{align*}
I_{m+n}^{\nu,\alpha}
  &=    \sum_{k=0}^{m}   (-1)^{k} \binom{m}{k}  I_{m-k}^{\nu,\alpha}   \pzk \left(  I_n^{\nu,\alpha}\right)
=  m! n!    \sum_{k=0}^{\min{(m,n)}}    \frac{ (2\alpha)^k}{k!}  \frac{ I_{m-k}^{\nu,\alpha}}{(m-k)!}  \frac{I_{n-k}^{\nu,\alpha}}{(n-k)!}.
 \end{align*}
 This completes the proof.
\end{proof}

\begin{remark} We recover from \eqref{NielsenIm}
	the three term recurrence formula \eqref{RecF11} satisfied by the polynomials $I_{n}^{\nu,\alpha}$ by taking $m=1$ with $n \geq 1$.
\end{remark}

\begin{remark}
	The most discussed algebraic results concerning the polynomials $ I_n^{\nu,\alpha} $,
	can be recovered easily by means of the well--known properties of the real Hermite polynomials $H_n(x)$ thanks to Corollary \ref{cor-expIm} or also Proposition \ref{propHH}. This is the case of the identity \eqref{Der-zxi} as well as
 Theorem \ref{thm-Nielson}, whose the proof
	 can also be handled by means of Corollary \ref{cor-expIm} above combined with the standard Nielsen's identity for the single real Hermite polynomials, or also using \eqref{GhInPoly-1}.
	The same observation holds true for Theorem \ref{GenFctExp} below.
	 However, the analytic properties of these polynomials are far from to be derived by employing  Corollary \ref{cor-expIm} as will be clarified in the following sections (see Sections 5, 6 and 7).
	This is the case of the integral representation of $ I_n^{\nu,\alpha} $ obtained in Theorem \ref{thm-IntRep} below.
\end{remark}

\section{Generating functions}

The first generating function we deal with is a standard one.

\begin{theorem}\label{GenFctExp}
	The polynomials $I_n^{\nu,\alpha}$ satisfy the generating function
	\begin{align} \label{GenFctIm1}
	\sum\limits_{n=0}^{+\infty} \frac{t^n}{n!} I_n^{\nu,\alpha}
	=    e^{\alpha t^2 + t I_1^{\nu,\alpha}} .
	\end{align}
\end{theorem}

\begin{proof}
	Notice first that we have
	\begin{align*}
	\sum\limits_{n=0}^{+\infty} \frac{t^n}{n!} I_n^{\nu,\alpha}
	&= e^{-\alpha z^2-\xi z} \sum\limits_{n=0}^{+\infty}\frac{\left(-t\right)^n }{n!} \left(\pz - \nu  \bz \right)^{n}\left(e^{\alpha z^2+\xi z}\right)
	\\& =e^{-\alpha z^2-\xi z}e^{-t\pz + \nu   t \bz }\left(e^{\alpha z^2+\xi z}\right)
	\\& =e^{-\alpha z^2-\xi z} e^{ \nu  t\bz }\exp\left({-t\pz }\right)\left(e^{\alpha z^2+\xi z}\right).
	\end{align*}
	Now, in view of Lemma \ref{lem-partial} and making appeal to the usual generating function of the Hermite polynomials (\cite{Rainville71,Szego75}):
	$$ \sum\limits_{k=0}^\infty  \frac{t^k}{k!} H_k(x)  = e^{-t^2 +2tx},$$
	it follows
	\begin{align*}
	\sum\limits_{n=0}^{+\infty} \frac{t^n}{n!} I_n^{\nu,\alpha}
	& =e^{ \nu  t\bz } \left( \sum\limits_{k=0}^{\infty}  \frac{(i\alpha^{1/2}t)^k}{k!}
	H_{k}\left(i\alpha^{1/2} z + \frac{i \xi}{2 \alpha^{1/2}}\right)  \right)
	 =  e^{ \alpha t^2  -  (2 \alpha  z - \nu \bz + \xi) t }.
	\end{align*}
	This ends the proof.
\end{proof}

The next generating function generalizes the previous one. Its proof is based essentially on the Nielsen's identity. Namely, we assert

\begin{theorem} We have
        \begin{align}\label{GenFctIm2}
                \sum\limits_{m,n=0}^{+\infty} \frac{u^mv^n}{m!n!} I_{m+n}^{\nu,\alpha}
                 &=  e^{\alpha (u+v)^2 + (u+v) I_1^{\nu,\alpha}}  .
        \end{align}
\end{theorem}

\begin{proof}
In view of \eqref{NielsenIm}, we can write the right hand--side of \eqref{GenFctIm2} as follows
        \begin{align*}
           \sum\limits_{m,n=0}^{+\infty}  \frac{u^mv^n}{m!n!} I_{m+n}^{\nu,\alpha}
           &=   \sum\limits_{m,n=0}^{+\infty} u^mv^n   \sum\limits_{k=0}^m \frac {(2\alpha )^k}{k!}
                           \frac{ I_{m-k}^{\nu,\alpha}}{(m-k)!}\frac{ I_{n-k}^{\nu,\alpha}}{(n-k)!}\\
                           &=   \sum\limits_{m=0}^{+\infty}  u^m \sum\limits_{k=0}^m \frac {(2\alpha v)^k}{k!}
                           \frac{ I_{m-k}^{\nu,\alpha}}{(m-k)!}
                            \left(\sum\limits_{n=0}^{+\infty}  \frac{v^{n-k} I_{n-k}^{\nu,\alpha}}{(n-k)!}\right)\\
                            &=   \sum\limits_{m=0}^{+\infty}  u^m \sum\limits_{k=0}^m \frac {(2\alpha v)^k}{k!}
                             \frac{I_{m-k}^{\nu,\alpha}}{(m-k)!}
                            \left(\sum\limits_{j=0}^{+\infty} v^j \frac{ I_{j}^{\nu,\alpha}}{j!}\right) .
                 \end{align*}
According to \eqref{GenFctIm1}, this leads to
            \begin{align*}
            \sum\limits_{m,n=0}^{+\infty} \frac{u^mv^n}{m!n!} I_{m+n}^{\nu,\alpha}
                             &=   \sum\limits_{m=0}^{+\infty} \sum\limits_{k=0}^m \frac {(2\alpha v)^k}{k!}
                            u^m \frac{ I_{m-k}^{\nu,\alpha}}{(m-k)!} e^{\alpha  v^2 + v I_1^{\nu,\alpha}} .
            \end{align*}
Now, by interchanging the order of summation in the double sum,
 $$\displaystyle \sum\limits_{m=0}^{+\infty} \sum\limits_{j=0}^m T_{j,m} = \sum\limits_{j=0}^{+\infty}  \sum\limits_{k=0}^\infty T_{j,j+k},$$
 it follows
           \begin{align*}
            \sum\limits_{m,n=0}^{+\infty} \frac{u^mv^n}{m!n!} I_{m+n}^{\nu,\alpha}
                           &=  \sum\limits_{k=0}^{+\infty}     \frac {(2\alpha uv)^k}{k!}
                           \sum\limits_{j=0}^\infty \frac{u^j I_{j}^{\nu,\alpha}}{j!} e^{\alpha  v^2 + v I_1^{\nu,\alpha}}.
            \end{align*}
Using again \eqref{GenFctIm1}, we obtain
            \begin{align*}
            \sum\limits_{m,n=0}^{+\infty} \frac{u^mv^n}{m!n!} I_{m+n}^{\nu,\alpha}
                 =  e^{2\alpha uv} e^{\alpha  u^2 + uI_1^{\nu,\alpha}} e^{\alpha  v^2 + vI_1^{\nu,\alpha}}
                =  e^{\alpha (u+v)^2 + (u+v) I_1^{\nu,\alpha}} .
            \end{align*}
\end{proof}

\begin{remark}
	For $u=0$ or $v=0$, the identity \eqref{GenFctIm2} reduces further to \eqref{GenFctIm1}.
\end{remark}

The third generating function in this section is the following

\begin{theorem}\label{genfct3}
We have the following identity
\begin{align} \sum\limits_{n=0}^{+\infty}\sum\limits_{k=0}^{m}\binom{m}{k} (i\alpha^{1/2})^{m-k}\frac{\xi^n}{\nu^n n!}H_{m-k}(i\alpha^{1/2}z)H^\nu_{n,k}(z,\overline{z})
=  e^{\xi z}I_m^{\alpha,\xi}(z).
 \end{align}
\end{theorem}

 \begin{proof} Direct computation making use the Leibniz formula infers
 \begin{align*}
 \frac{\partial^m}{\partial z^m}(e^{-\nu|z|^2+\xi z}e^{\alpha z^2})
 &=\sum\limits_{k=0}^{m}\binom{m}{k}\frac{\partial^k}{\partial z^k} (e^{-\nu|z|^2+\xi z}) \frac{\partial^{m-k}}{\partial z^{m-k}}(e^{\alpha z^2}).
 \end{align*}
 By expanding $e^{\xi z}$ in power series and making use of Definition \ref{HolHm} of the holomorphic Hermite polynomials, we get
 \begin{align*}
 \frac{\partial^m}{\partial z^m}(e^{-\nu|z|^2+\xi z}e^{\alpha z^2})
 &=\sum\limits_{k=0}^{m}\binom{m}{k}\frac{\partial^k}{\partial z^k}\left(\sum\limits_{n=0}^{+\infty}\frac{\xi^n}{n!}z^ne^{-\nu |z|^2}\right)(-i\alpha^{1/2})^{m-k}H_{m-k}(i\alpha^{1/2}z)e^{\alpha z^2}\\
 &=e^{\alpha z^2}\sum\limits_{n=0}^{+\infty}\frac{\xi^n}{n!\nu^n} \left( \sum\limits_{k=0}^{m}\binom{m}{k}(-i\alpha^{1/2})^{m-k}\frac{\partial^k}{\partial z^k}(\nu^n z^ne^{-\nu|z|^2}) \right)H_{m-k}(i\alpha^{1/2}z)\\
 &=(-1)^m e^{-\nu|z|^2+\alpha z^2}\sum\limits_{n=0}^{+\infty}\frac{\xi^n}{n!\nu^n}\sum\limits_{k=0}^{m}\binom{m}{k}(i\alpha^{1/2})^{m-k} e^{\nu|z|^2}H^\nu_{n,k}(z,\overline{z})H_{m-k}(i\alpha^{1/2}z).
\end{align*}
Therefore,
 $$ I_m^{\alpha,\xi}(z)
  =e^{-\xi z}\sum\limits_{n=0}^{+\infty}\sum\limits_{k=0}^{m}\binom{m}{k} (i\alpha^{1/2})^{m-k}\frac{(\xi)^n}{\nu^n n!}H_{m-k}(i\alpha^{1/2}z)H^\nu_{n,k}(z,\overline{z}).$$
 \end{proof}

The last generating function in this section shows that the polynomials $I_m^{\nu,\alpha} (z,\bz|\xi)$ can be generated from the $\xi$--holomorphic Hermite polynomials $H_m(\xi)$ and the polyanalytic Hermite polynomials $ H_{m,n}(z,\bz)$.
To this end, we use a variant (analytic continuation) of the generating function of the real Hermite polynomials.
Such result is also needed in the proof of Theorem \ref{genfct3} below.

\begin{lemma}\label{lem-partial} The explicit expression of the $k$--th $z$--derivative of $e^{\alpha z^2+\xi z}$ in terms of the usual Hermite polynomials $H_{k}(z)$ is given by
	\begin{align}\label{partial-k}
	\pzk e^{\alpha z^2+\xi z} &=  (- i)^k \alpha^{k/2}   H_{k}\left(i\alpha^{1/2} z + \frac{i \xi}{2 \alpha^{1/2}}\right)  e^{\alpha z^2+\xi z}.
	\end{align}
	Subsequently, we have
	\begin{align}\label{expvarphi0}
	e^{\alpha z^2+\xi z}=\sum\limits_{n=0}^{+\infty} \frac{(-i)^{n}}{n!}\alpha^{\frac{n}{2}}H_n\left(\frac{i\xi}{2\alpha^{1/2}}\right)z^n.
	\end{align}
\end{lemma}

\begin{theorem} \label{thm-ImBilGenFct} We have the generating function
     \begin{align}\label{ImBilGenFct}
   \sum_{k=0}^\infty   \frac{(-i)^k \alpha^{k/2}}{\nu^k k!}  H_{k}\left( \frac{i \xi}{2 \alpha^{1/2}}\right)  H^\nu_{k,m}(z,\bz)
   =    I_m^{\nu,\alpha} (z,\bz|\xi)e^{\alpha z^2 + \xi z}  .
     \end{align}
\end{theorem}

\begin{proof} The proof follows easily starting from the definition of $I_n^{\nu,\alpha}$ and using the expansion series of the entire function $e^{\alpha z^2+\xi z} $ as
	in \eqref{expvarphi0}. Indeed, we get
	\begin{align}
	I_n^{\nu,\alpha}(z,\bz|\xi)
	&=  e^{-\alpha z^2-\xi z} \sum\limits_{m=0}^{+\infty} \frac{(-i)^{m} }{m!} \alpha^{\frac{m}{2}}
	H_m\left(\frac{i\xi}{2\alpha^{1/2}}\right)  (-1)^{n} e^{\nu|z|^2}
	\partial^n_z ( z^m e^{-\nu|z|^2} ) \nonumber
	\\&=e^{-\alpha z^2-\xi z} \sum\limits_{m=0}^{+\infty}  \frac{\left(-\frac{i\alpha^{1/2} }{\nu}\right)^{m}}{m!}
	H_m\left(\frac{i\xi}{2\alpha^{1/2}}\right) H^\nu_{m,n}(z,\bz)  \label{bilinear}.
	\end{align}
The last equality follows by observing that the rescaled complex Hermite polynomials $H^\nu_{m,k}(z,\bz)$ can be represented also as
$ H^\nu_{k,m}(z,\bz) = (-1)^m e^{\nu |z|^2} \pzm (z^k  e^{-\nu |z|^2} ) .$
\end{proof}

\begin{remark}
	The identity \eqref{ImBilGenFct} states that the polynomials $I_n^{\nu,\alpha}(z,\bz|\xi)$ appear as the bilinear generating function of the polynomials $H^\nu_{m,n}$ and $H_n$. This fact can be used to recover the result of Corollary \ref{cor-expIm} giving the explicit expression of $I_n^{\nu,\alpha}(z,\bz|\xi)$  to Theorem 2.1 in \cite{BenahmadiG2018}.
\end{remark}

\section{Orthogonality}

We begin by considering the case of $\xi=0$.

 \begin{theorem}\label{thm-orthogImtensor}
	Let $\nu>0$ and $\alpha\in \R$ such that $2|\alpha|<\nu$.
	Then, the polynomials $I_m^{\nu,\alpha}(z,\bz|0)$ satisfy the orthogonality property
	\begin{align}  \label{Orthog}
	\int_{\C} I_m^{\nu,\alpha}(z,\bz|0) \overline{I_n^{\nu,\alpha}(z,\bz|0)}
		e^{-\nu |z|^2 +\alpha  (z^2+\bz^2)}  d\lambda(z)
	= \frac{\pi \nu^n n! }{\sqrt{\nu^2-4\alpha^2}}  \delta_{n,m}.
	\end{align}
\end{theorem}

\begin{proof}
	Under the assumption $2|\alpha| <\nu$ and keeping in mind the result of Proposition \ref{propHH} as well as the orthogonality of the rescaled real Hermite polynomials $H^\tau_k$ in the Hilbert space $L^{2}(\R,e^{-\tau t^2}dt)$,
	$$ \int_{\R}  H^\tau_j(t) H^\tau_k(t) e^{-\tau t^2}dt =
	\left( \frac{\pi}{\tau}\right)^{1/2} 2^{k} \tau^{k} k!,$$
	 we get
	\begin{align*}
		\int_{\C} & I_m^{\nu,\alpha}(z,\bz|0) \overline{I_n^{\nu,\alpha}(z,\bz|0)}
	e^{-\nu |z|^2 +\alpha  (z^2+\bz^2)}  d\lambda(z)
	\\&= 		\int_{\C} I_m^{\nu,\alpha}(z,\bz|0)  \overline{I_n^{\nu,\alpha}(z,\bz|0)}
	e^{-(\nu-2\alpha)x^2 - (\nu+2\alpha)y^2}  dxdy
		\\&	=  \sum_{j=0}^m	\sum_{k=0}^n \frac{(-i)^j (i)^k}{2^{m+n}} \binom{m}{j}   \binom{n}{k}
	\norm{H^{\nu-2\alpha}_{n-k}}^2_{L^{2,\nu-2\alpha}(\R)}   	\norm{H^{\nu+2\alpha}_k}^2_{L^{2,\nu+2\alpha}(\R)}  \delta_{m-j, n-k} \delta_{j,k}
			\\&	= \frac{1}{2^{m+n}} 	\sum_{k=0}^{\min(m,n)}  \binom{m}{k}  \binom{n}{k}
	\norm{H^{\nu-2\alpha}_{n-k}}^2_{L^{2,\nu-2\alpha}(\R)}   	\norm{H^{\nu+2\alpha}_k}^2_{L^{2,\nu+2\alpha}(\R)}  \delta_{m-k, n-k}
				\\&	=   \frac{\pi}{\sqrt{\nu^2-4\alpha^2}}  \frac{n!}{2^{n}}  	\sum_{k=0}^{n}   \binom{n}{k}
(\nu-2\alpha)^{n-k}   (\nu+2\alpha)^{k}   \delta_{m, n}
				\\&	=   \frac{\pi \nu ^{n} n! } {\sqrt{\nu^2-4\alpha^2}}  \delta_{m, n} .
	\end{align*}
	This completes the proof.
	\end{proof}

\begin{remark}
		For $\alpha=0$ with $\nu>0$, we recover the classical orthogonality for the monomials $I_n^{\nu,0}(z,\bz|0)= \nu^n \bz^n$.
	\end{remark}

\begin{remark}
The proof we have furnished above is also valid for the general case of arbitrary $\xi$ under the assumption that $2|\alpha|<\nu$.
\end{remark}

Based on the orthogonal property obtained in \cite{Eijndhoven-Meyers1990}  for the holomorphic Hermite polynomials $ H_n(z)$, to wit
\begin{align}\label{VanEM90}
\int_{\R^2} H_m(x + iy)H_n(x - iy) e^{-(1-\theta)x^2- \left( \frac{1}{\theta}-1\right)  y^2} dx dy = \frac{ \theta^{1/2}\pi}{1-\theta}  \left(\frac{2(1+\theta)}{1-\theta}\right)^n n!\delta_{m,n},
\end{align}
where $0 < \theta < 1$, we can deduce two orthogonality relations for the polynomials $I_n^{0,\alpha}(z,\bz|0)$ corresponding to $\nu=0=\xi$ according to $\alpha>0$ or $\alpha<0$. The one for $\alpha>$ reads
\begin{align}\label{VanEM90I0}
\int_{\R^2} I_n^{0,\alpha}(z,\bz|0) \overline{I_m^{0,\alpha}(z,\bz|0)}   e^{- \left( \frac{1}{\theta}-1\right)x^2-\alpha(1-\theta)y^2} dx dy = \frac{ \theta^{1/2}\pi}{\alpha(1-\theta)}  \left(\frac{2\alpha (1+\theta)}{1-\theta}\right)^n n!\delta_{m,n},
\end{align}
since in this case $I_n^{0,\alpha}(z,\bz|0) = (i\sqrt{\alpha})^n H_n(i\sqrt{\alpha} z)$.

We establish below an orthogonal property for $I_n^{\nu,\alpha}(z,\bz|\xi)$ for arbitrary $\nu>0$ and $\xi \in \C$,  generalizing \eqref{Orthog} as well as \eqref{VanEM90}.
 To this end, for given reals $a,b>0$, we consider the weight function
\begin{align*}
\omega^{a,b}_{\nu,\alpha,\xi}(z,\bz)
= e^{-A_{\nu,\alpha}^{a,b} |z|^2 - B_{\nu,\alpha}^{a,b}\left( z^2+\bz^2\right) +2\Re(C_{\nu,\alpha,\xi}^{a,b} z )}  e^{-a\Re(\xi)^2-b\Im(\xi)^2}
\end{align*}
where the quantities $A_{\nu,\alpha}^{a,b}$, $B_{\nu,\alpha}^{a,b}$ and $C^{\nu,\alpha,\xi}_{a,b}$ are given by
	\begin{align*}
&A_{\nu,\alpha}^{a,b} := \frac{(\nu-2\alpha)^2a +  (\nu+2\alpha)^2b}2
\\& B_{\nu,\alpha}^{a,b} := \frac{(\nu-2\alpha)^2a - (\nu+2\alpha)^2b}4
\\& C^{\nu,\alpha,\xi}_{a,b} := a(\nu-2\alpha)\Re(\xi) +ib (\nu+2\alpha)\Im(\xi).
\end{align*}

\begin{theorem}\label{thm-orthogIm2} Let $a,b>0$ such that $4\alpha ab=a-b$. Then, the polynomials $I_n^{\nu,\alpha}(z,\bz|\xi)$ satisfy the orthogonality property
	\begin{align}\label{OrthP2}
	 \int_{\C} I_m^{\nu,\alpha}(z,\bz|\xi)   \overline{I_n^{\nu,\alpha}(z,\bz|\xi) } \omega^{a,b}_{\nu,\alpha,\xi}(z,\bz) d\lambda(z)
&	= \frac{\pi }{\sqrt{ab}|\nu^2-4\alpha^2|}   \left(\frac{a+b}{2ab}\right)^n n!\delta_{m,n}.
\end{align}
\end{theorem}

\begin{proof}
Theorem \ref{GenFctExp} yields
	 \begin{align*}
	S_{m,n}^{\nu,\alpha,\xi}(u,v|z,\bz)&=\sum\limits_{m,n}^{+\infty}\frac{u^m v^n}{m!n!} I_m^{\nu,\alpha}(z,\bz|\xi)   \overline{I_n^{\nu,\alpha}(z,\bz|\xi) }
	\\&=e^{\alpha( u^2 +v^2) + u I_1^{\nu,\alpha}(z,\bz|\xi) + v \overline{I_1^{\nu,\alpha}(z,\bz|\xi)}}
	\\&  =e^{\alpha(u^2+v^2) + [(\nu-2\alpha)x -\Re(\xi)](u+v) -i[(\nu+2\alpha)y +\Im(\xi)](u-v)}
	\\&  =e^{\alpha(u^2+v^2) + (u+v)X -i(u-v)Y},
	\end{align*}
	where we have set $X=(\nu-2\alpha)x-\Re(\xi)$ and $Y=(\nu+2\alpha)y+\Im(\xi)$
for given $z=x+iy$.
 Now, if we denote the left--hand side of \eqref{OrthP2} by $T_{m,n}^{\nu,\alpha}(\xi)$, then
 \begin{align*} T_{m,n}^{\nu,\alpha}(\xi)
  &= \frac{1}{|\nu^2-4\alpha^2|}\int_{\R^2}I_m^{\nu,\alpha}(z,\bz|\xi)   \overline{I_n^{\nu,\alpha}(z,\bz|\xi) } \omega^{a,b}_{\nu,\alpha,\xi}(z,\bz) dXdY
\end{align*}
with $z=z(X,Y)$ and $\bz=\overline{z(X,Y)}$.
Subsequently, we have
\begin{align*}
\sum\limits_{m,n=0}^{+\infty}
\frac{u^m v^n}{m!n!} T_{m,n}^{\nu,\alpha}(\xi)
&= \frac{1}{|\nu^2-4\alpha^2|}\int_{\R^2} S_{m,n}^{\nu,\alpha,\xi}(u,v|z,\bz) \omega^{a,b}_{\nu,\alpha,\xi}(z,\bz)   dXdY
\\
&= \frac{e^{\alpha(u^2+v^2)} }{|\nu^2-4\alpha^2|}
\int_{\R^2}  e^{-aX^2 + (u+v)X -b Y^2 -i(u-v)Y } dXdY
\\&= \frac{\pi }{\sqrt{ab}|\nu^2-4\alpha^2|}
e^{\frac{4\alpha ab+b-a}{4ab} (u^2+v^2)}
e^{ \frac{a+b}{2ab}uv}
\\&
= \frac{\pi }{\sqrt{ab}|\nu^2-4\alpha^2|} .
e^{ \frac{a+b}{2ab}uv},
\end{align*}
The third equality is obtained making use of the well--known Gaussian integral
\begin{align}\label{GaussInt0}
\int_{\R} e^{-\tau y^2 + \zeta y }dy =  \left(\frac{\pi}{\tau}\right)^{1/2} e^{\frac{\zeta^2}{4\tau}}; \qquad \tau>0, \, \zeta\in \C,
\end{align}
 while the last equality readily follows since $4\alpha ab+b -a=0$.
 Subsequently, we obtain
$$  S_{m,n}^{\nu,\alpha}(\xi)
= \frac{\pi }{\sqrt{ab}|\nu^2-4\alpha^2|}  \left(\frac{a+b}{2ab}\right)^n n!\delta_{m,n}.
$$
This completes our check of \eqref{OrthP2}.
\end{proof}

\begin{remark}
	As example of pairs $(a,b)$; $a,b>0$ satisfying the condition
	$4\alpha ab =a-b$, we can consider $a=(\nu-2\alpha)^{-1}$ and $b=(\nu+2\alpha)^{-1}$.
	The corresponding $A_{\nu,\alpha}^{a,b}$, $B_{\nu,\alpha}^{a,b}$ and $ C^{\nu,\alpha,\xi}_{a,b}$ are given by
	$A_{\nu,\alpha}^{a,b}=\nu$, $B_{\nu,\alpha}^{a,b}=- \alpha$ and $C^{\nu,\alpha,\xi}_{a,b}=\xi$, so that the orthogonality \eqref{OrthP2} reduces to
		\begin{align*}
	\int_{\C} I_m^{\nu,\alpha}(z,\bz|\xi)   \overline{I_n^{\nu,\alpha}(z,\bz|\xi) }
	 e^{-\nu|z|^2 +\alpha \left( z^2+\bz^2\right) +2\Re(\xi z )}   d\lambda(z)
	&	= \frac{\pi  \nu^n n!}{\sqrt{\nu^2-4\alpha^2}}    e^{\nu|\xi|^2 -\alpha \left( \xi^2+\overline{\xi}^2\right)} \delta_{m,n}.
	\end{align*}
	which for $\xi=0$ leads to the one obtained in Theorem \ref{thm-orthogImtensor}.
\end{remark}

\begin{remark}
	For $\nu=0=\xi$ and $\alpha>0$, the identity \eqref{OrthP2} reduces further to \eqref{VanEM90I0} by taking
	$$ a=\frac 1{4\alpha} \left( \frac 1\theta -1 \right) \quad \mbox{and} \quad b=\frac 1{4\alpha} \left( 1 -\theta \right)$$
	with $0<\theta<1$.
\end{remark}

\section{Integral representations}

In virtue of Theorem \ref{thm-RodF}, we obtain the following integral representation of the polynomials $I_n^{\nu,\alpha}(z,\bz|\xi )$.

\begin{proposition} For every $\nu>0$ and $\alpha \in\R$ with $\alpha\ne 0$, we have
	\begin{align}\label{IntRep1x}
	I_n^{\nu,\alpha}(z,\bz|\xi)&= \left( \frac{1}{\alpha \pi} \right)^{1/2} 2^n \int_{\R} t^n e^{-\frac 1 {4\alpha} \left( 2t - I_{1}^{\nu,\alpha}(z,\bz|\xi )\right)^2} dt .
		\end{align}
	\end{proposition}

\begin{proof} By means of the explicit formula for the gaussian integral \eqref{GaussInt0}, we can write
	\begin{align}
	e^{\frac{\left(I_{1}^{\nu,\alpha}(z,\bz|\xi )\right)^2}{4\alpha} }=\left(\frac{\alpha}{\pi}\right)^{\frac{1}{2}} \int_{\R}e^{-\alpha t^2 + tI_{1}^{\nu,\alpha}(z,\bz|\xi )}dt.
	\end{align}	
	The integral in the right--hand side converges uniformly on every disc $D(0,r) \subset \C$ and one can repeatedly differentiate it with respect to $z$. Hence, by \eqref{RodF} we obtain 	
	\begin{align*}
	I_n^{\nu,\alpha}(z,\bz|\xi )
	&= (-1)^n \left(\frac{\alpha}{\pi}\right)^{\frac{1}{2}} e^{-\frac{\left(I_{1}^{\nu,\alpha}(z,\bz|\xi )\right)^2}{4\alpha} }
	\int_{\R} \dfrac{\partial^n}{\partial z^n} e^{-\alpha t^2 + tI_{1}^{\nu,\alpha}(z,\bz|\xi )}dt \\
	&= (-1)^n \left(\frac{\alpha}{\pi}\right)^{\frac{1}{2}}
	\int_{\R} (-2\alpha t)^n e^{-\frac{1}{\alpha}\left( (\alpha t)^2 - \alpha tI_{1}^{\nu,\alpha}(z,\bz|\xi ) +
		\frac{\left(I_{1}^{\nu,\alpha}(z,\bz|\xi )\right)^2}{4}\right) }dt \\
	&= \left(\frac{1}{\alpha \pi}\right)^{\frac{1}{2}}
	\int_{\R} u^n e^{-\frac{1}{\alpha}\left( u - \frac{I_{1}^{\nu,\alpha}(z,\bz|\xi )}{2} \right)^2 }du .
	\end{align*}
	This completes our check of \eqref{IntRep1x}.
\end{proof}

The next result is a consequence of Theorem \ref{thm-ImBilGenFct} combined with the integral representation of the complex Hermite polynomials.

\begin{theorem} \label{thm-IntRep}
For $\nu>0$ and $a,b\in \C$ such that $ab>0$, we have
\begin{align}\label{IntRep1}
I_n^{\nu,\alpha}(z,\bz|\xi) &=
\left(\frac{ab}{\nu\pi}\right)
e^{\nu |z|^2-\alpha z^2-\xi z} \int_{\C}  (b \overline{\zeta})^n e^{ -\frac{ab}{\nu}|\zeta|^2 + \frac{ a^2 \alpha  }{\nu^2}\zeta^2  - \frac{ a \xi }{\nu} \zeta}
e^{a \zeta \bz - b \bzeta z } d\lambda(\zeta) .		
\end{align}
More particularly, we have
	\begin{align}\label{IntRep2}
	I_n^{\nu,\alpha}(z,\bz|\xi) &= \left(\frac{\nu}{\pi}\right) e^{\nu|z|^2 -\alpha z^2 -\xi z}
	 \int_{\C} (-\nu \overline{\zeta})^n  e^{-\nu |\zeta|^2 + \alpha \zeta^2 + \xi \zeta} e^{2 i\nu \Im\scal{z,\zeta}} d\lambda(\zeta).
	\end{align}
\end{theorem}

\begin{proof}
The result follows by a tedious but straightforward computation. In fact, starting from Theorem \ref{thm-ImBilGenFct}
	and using the integral representation of $H^\nu_{m,n}(z,\bz)$ given by Theorem 3.1 in \cite{BenahmadiG2018}, to wit
	\begin{eqnarray}\label{intHermite}
	H_{m,n}^{\nu}(z;\bz) = \left(\frac{ab}{\nu\pi}\right)  (-a)^m(b)^n
	\int_{\C} \zeta^m \overline{\zeta}^n e^{\nu |z|^2 -\frac{ab}{\nu}|\zeta|^2 +a \zeta \bz - b \bzeta z} d\lambda(\zeta)
	\end{eqnarray}
	(valid for $\nu>0$ and $a,b\in \C$ such that $ab>0$), we obtain
		\begin{align*}
		I_n^{\nu,\alpha}(z,\bz|\xi)  &=
		\left(\frac{ab}{\nu\pi}\right)
		e^{-\alpha z^2-\xi z} \int_{\C}  (b\overline{\zeta})^n e^{\nu |z|^2 -\frac{ab}{\nu}|\zeta|^2 +a \zeta \bz - b \bzeta z}
	\left(\sum\limits_{m=0}^{+\infty}   \frac{\left(\frac{-i a \alpha^{1/2} \zeta }{\nu}\right)^{m}}{m!}
		H_m\left(-\frac{i\xi}{2\alpha^{1/2}}\right)\right) d\lambda(\zeta)
		\\ &=
\left(\frac{ab}{\nu\pi}\right)
e^{\nu |z|^2-\alpha z^2-\xi z} \int_{\C}  (b \overline{\zeta})^n e^{ -\frac{ab}{\nu}|\zeta|^2 + \frac{ a^2 \alpha  }{\nu^2}\zeta^2  - \frac{ a \xi }{\nu} \zeta}
e^{a \zeta \bz - b \bzeta z } d\lambda(\zeta) .		
		\end{align*}
The particular case of $a=b=-\nu$ gives rise to \eqref{IntRep2}. This completes the proof.
\end{proof}

\begin{remark}
	The obtained result \eqref{IntRep2} can also reproved directly.
	Indeed, by rewriting  $e^{-\nu |z|^2+\alpha z^2+\xi z} $ as
	\begin{align*}
	e^{-\nu |z|^2+\alpha z^2+\xi z} &=e^{-\frac{\nu^2}{\nu+\alpha}\left(x-\frac{\xi}{2\nu}\right)^2+(\nu+\alpha)\left(iy+\frac{2\alpha x+\xi}{2(\nu+\alpha)}\right)^2}
	\end{align*}
	and next using twice the integral representation of the Gaussian function \eqref{GaussInt0},
	 we obtain the integral representation of  $e^{-\nu |z|^2+\alpha z^2+\beta z} $, to wit
		\begin{align*}
	e^{-\nu |z|^2+\alpha z^2+\xi z} &=\frac{1}{4\pi \nu} \int_{\R^2}  e^{-\frac{1}{4 \nu^2} \left((\nu-\alpha) t^2 + (\nu+\alpha) s^2\right) +i(yt+xs) + \frac{\xi}{2\nu} (t - i s) - \frac{i\alpha }{2\nu^2} ts } dt ds
	\end{align*}
  under the assumption that $\nu+\alpha>0$ with $z=x+iy$. It can be rewritten in the form
		\begin{align}\label{IntRep3}
	e^{-\nu |z|^2+\alpha z^2+\xi z}  &=\frac{1}{2\pi} \int_{\C}  e^{-\nu |\zeta|^2 + \alpha\zeta^2 +\xi \zeta +2i\nu \Im \scal{z,\zeta}}d\lambda(\zeta)
	\end{align}
	making the change $	\zeta= \frac{t-is}{ 2\nu}$.
Thus \eqref{IntRep2}  follows readily by derivation of \eqref{IntRep3}.
\end{remark}

We conclude this section by realizing the polynomials $I_n^{\nu,\alpha}(z,\bz|\xi)$ as the image of the real Hermite function $h^\nu_n(t) 
=\sqrt{\nu}^n e^{-\frac{\nu}{2}t^{2}}H_{n}(\sqrt{\nu}t)$ by
rescaled Fourier--Wigner transform defined on $L^2(\R)$ by
\cite{Wong98,Folland89,Thangavelu98,deGosson2017}
\begin{align*}
\mathcal{V}^\nu(f,g)(x,y)
=\left( \frac{\nu}{2\pi}\right)^{1/2} \int_{\R}e^{i\nu\left(t+\frac{x}{2}\right)y} f(t+x)\overline{g(t)}dt; z=x+iy,
\end{align*}
with respect to a special window function $g$ that we determine explicitly. Thus, we define
\begin{align*}
\mathcal{W}^{\nu}_{\alpha,\xi}(f)(z,\bz):=
\frac{(-1)^{n}}{2^{n}} \left( \frac{2\nu}{\nu + 2\alpha} \right)^{1/2}  e^{\frac{\nu}2|z|^2}
\mathcal{V}^{2\nu} \left( M^{\nu}_{\alpha}  , f\right)  (x,y),
\end{align*}
where $M^{\nu}_{\alpha}$ stands for
\begin{align}\label{g}
M^{\nu}_{\alpha}(y):=e^{ -\frac{\xi^2}{2(\nu +  2\alpha)}   } \exp\left( - \frac{\nu}{\nu +  2\alpha}  \left( (\nu -  2\alpha)  y^2 - 2\xi y \right) \right).
\end{align}
More explicitly,
\begin{align}\label{IntFWI}
\mathcal{W}^{\nu}_{\alpha,\xi}(f)(z,\bz)
&=   \frac{(-1)^{n}}{2^{n}} \left( \frac{2\nu^2}{(\nu + 2\alpha)\pi} \right)^{1/2}  e^{\frac{\nu}{2}|z|^2-\alpha z^2- \xi z} e^{ -\frac{\xi^2}{2(\nu +  2\alpha)}   }
\\& \times\int_{\R}  e^{2i\nu\left( t-\frac x2\right) y}
\exp\left( - \frac{\nu}{\nu +  2\alpha} \left( (\nu -  2\alpha) t^2 -2t\xi \right)    \right)  f(t-x) dt. \nonumber
\end{align}

\begin{theorem} Let $\nu$ and $\alpha$ be such that $2|\alpha|<\nu$.  Then, for every $z$, we have
	\begin{align*}
	\mathcal{W}^{\nu}_{\alpha,\xi}(h^{2\nu}_n)(z,\bz) = I_n^{\nu,\alpha} (z,\bz|\xi).
	\end{align*}
\end{theorem}

\begin{proof}
	Observe first that the polynomials $H^{\tau}_{m,n}(z,\bz )$ can be realized as
	\begin{eqnarray}\label{actionV}
	H^{\nu}_{m,n} (z, \bz) = (-1)^{n} \frac{\sqrt{2}}{2^{m+n}}   e^{\frac{\nu}2|z|^2} \mathcal{V}^{2\nu}(h^{2\nu}_m,h^{2\nu}_n) (x,y) ; \quad z=x+iy .
	\end{eqnarray}
	This follows by straightforward computation using Theorem 3.1 in \cite{ABEG2015} as well as the fact that $ h^\tau_{m,n}(z,\bz) := \tau^{(m+n)/2}
	h_{m,n}(\tau^{1/2}z,\tau^{1/2}\bz)$.
	Now, by Theorem \ref{thm-ImBilGenFct}, we get
	\begin{align*}
	I_n^{\nu,\alpha} (z,\bz|\xi)e^{\alpha z^2 + \xi z}
	&= (-1)^{n} \frac{\sqrt{2}}{2^{n}} e^{\frac{\nu}2|z|^2}
	\sum_{k=0}^\infty  \frac{(-i)^k \alpha^{k/2}}{2^{k}\nu^k k!}  H_{k}\left( \frac{i \xi}{2 \alpha^{1/2}}\right)   \mathcal{V}^{2\nu}(h^{2\nu}_k,h^{2\nu}_n) (x,y)
	\nonumber\\
	&= (-1)^{n} \frac{\sqrt{2}}{2^{n}} e^{\frac{\nu}2|z|^2}
	\mathcal{V}^{2\nu} \left(  \sum_{k=0}^\infty  \frac{(-i)^k \alpha^{k/2}}{2^{k}\nu^k k!}  H_{k}\left( \frac{i \xi}{2 \alpha^{1/2}}\right) h^{2\nu}_k  ,h^{2\nu}_n\right)  (x,y).
	\end{align*}
	Making use of the Mehler's formula
	(\cite{Mehler1866,Rainville71}) for the rescaled Hermite functions $h^\tau_n$, to wit
	\begin{align}\label{Mehlerkernelhn}
	\sum\limits_{k=0}^{+\infty} \frac{ \lambda^k h^{\tau}_{k}(X)  h^{\tau}_{k}(Y)}{  2^k  \tau^k k!}= \frac{1}{\sqrt{1 - \lambda^2}}  \exp\left( - \frac{\tau (1+\lambda^2)}{2(1-\lambda^2)} (X^2 + Y^2) + \frac{2\tau \lambda}{1 - \lambda^2} XY \right)
	\end{align}
	valid for $|\lambda|<1$, with $\tau=2\nu$, $ X = \frac{i \xi}{2 (2\nu\alpha)^{1/2}}$ and $ \lambda = -i\left( \frac{ 2\alpha}{\nu} \right) ^{1/2} $, we get
	\begin{align*}
	\sum_{k=0}^\infty  \frac{(-i)^k \alpha^{k/2}}{2^{k}\nu^k k!}  H_{k}\left( \frac{i \xi}{2 \alpha^{1/2}}\right) h^{2\nu}_k (Y)
	&=     e^{ -\frac{\xi^2}{2.4 \alpha} }  \sum_{k=0}^\infty  \frac{\left( -i\left( \frac{ 2\alpha}{\nu} \right) ^{1/2}\right)^k }{2^{k} (2\nu)^k  k!}
	h^{2\nu}_{k}\left( \frac{i \xi}{2 (2\nu\alpha)^{1/2}} \right)
	h^{2\nu}_k(Y)\\
	&=	\left( \frac{\nu}{\nu + 2\alpha} \right)^{1/2} e^{ -\frac{\xi^2}{2(\nu +  2\alpha)}   }
	\exp\left( - \frac{\nu \left( (\nu -  2\alpha)  Y^2 -  2\xi Y \right)}{\nu +  2\alpha}  \right)
	\\  &=	\left( \frac{\nu}{\nu + 2\alpha} \right)^{1/2} M^{\nu}_{\alpha}(Y),
	\end{align*}
	where $M^{\nu}_{\alpha} $ is exactly the function given through \eqref{g} under the assumption that $2|\alpha|<\nu$.
	Therefore, we arrive at
	\begin{align*}
	I_n^{\nu,\alpha} (z,\bz|\xi)e^{\alpha z^2 + \xi z}
	&=  \frac{(-1)^{n}}{2^{n}} \left( \frac{2\nu}{\nu + 2\alpha} \right)^{1/2}  e^{\frac{\nu}2|z|^2}
	\mathcal{V}^{2\nu} \left( M^{\nu}_{\alpha}  ,h^{2\nu}_n \right)  (x,y)	.
	\end{align*}
	The obtained expression is reads equivalently as \eqref{IntFWI}.
	This completes the proof.
\end{proof}

\begin{remark}
	For the particular case of $\alpha=0=\xi$ and $\nu=1/2$, the transform $\mathcal{W}^{\nu}_{\alpha,\xi}$ in \eqref{IntFWI} reduces further to the Segal--Bargmann transform $\mathcal{B}$ from
	$L^{2}(\R; dt)$ onto the Bargmann space $\mathcal{F}^{2,1/2}(\C)=\mathcal{H}ol(\C) \cap L^{2}(\C;e^{-\frac{|z|^2}2}dxdy)$. In fact, we have
	\begin{align*}
	\mathcal{W}^{1/2}_{0,0}(h_n)(z,\bz)
	&=  \frac{(-1)^{n}}{2^n\sqrt{\pi}}  e^{\frac{\nu}{4}|z|^2}
	\int_{\R}  e^{\frac i2 \left( 2t+x\right) y}
	e^{-\frac 12 (t+x)^2 }  h_n(t) dt\\
	&=  \frac{(-1)^{n}}{2^n\sqrt{\pi}}
	\int_{\R} e^{-\frac 12 \left( t^2+2t \bz +\frac{\bz^2}2 \right)  }  h_n(t) dt\\
	&= \frac{(-1)^{n}}{2^n} \mathcal{B}(h_n)(-\bz),
	\end{align*}
	so that the result of our Theorem which reads $\mathcal{W}^{\nu}_{\alpha,\xi}(h_n)(z,\bz)= I_n^{1/2,0} (z,\bz|0)= (1/2)^n \bz^n$ is exactly the reproducing property for the monomials by $\mathcal{B}$, $\mathcal{B}(h_n)(-\bz)=(-1)^n\bz^n$.
\end{remark}

 \section{Polyanalyticity and partial differential equations } \label{Spoly}

The introduced polynomials is a special subclass of polyanalytic functions on the complex plane. In counterpart of the $H^\nu_{m,n}(z,\bz)$ which are polyanalytic of order $n+1$ and anti--polyanalytic of order $m+1$, the polyanlyticity and the anti--polyanalyticity of the polynomials $I_n^{\nu,\alpha} (z,\bz|\xi)$ have the same order. This is due to the fact that
 $$\pbz^{n+1} I_n^{\nu,\alpha} (z,\bz|\xi)=0 = \pz^{n+1} I_n^{\nu,\alpha} (z,\bz|\xi),$$
  which can be handled easily using \eqref{DerbarIm} and \eqref{DerImExp} keeping in mind the fact $I_0^{\nu,\alpha} (z,\bz|\xi)=1$. Indeed, by induction we have
\begin{align}\label{DerbarImG}
\pbz^k I_n^{\nu,\alpha} (z,\bz|\xi)=  \frac{n!\nu ^k}{(n-k)!} I_{n-k}^{\nu,\alpha}
\quad \mbox{and} \quad
\pz^k  I_n^{\nu,\alpha}=  \frac{(-2\alpha)^k n!}{(n-k)!} I_{n-k}^{\nu,\alpha}
\end{align}
for every nonnegative integer $k\leq n$. This can also be recovered from \eqref{RodriguesIm}, since the polyanalyticity of a complex--valued function $f$ is equivalent to $f$ be of the form
$$
f(z,\bz) = (-1)^{n} e^{\nu   z \bz }  \pzn \left(e^{- \nu   z \bz } h \right)
$$
for some nonnegative integer $n$ and holomorphic function $h$ (see e.g. \cite{Balk1991,AbreuFeichtinger2014}).
 Subsequently, by means of \cite{Burgatti1922} there exist certain holomorphic functions $h_{k}$; $k=0,1,\cdots,n$ such that    \begin{align}\label{ExpPolBzI}
 I_n^{\nu,\alpha} (z,\bz|\xi)= \bz^{n} h_{n}+ \cdots + \bz h_{1} + h_{0}.
  \end{align}
   The next result gives the explicit expressions of the holomorphic component $h_k$ of $I_n^{\nu,\alpha}$.

\begin{theorem} \label{thm-ImHk} The polynomials $I_n^{\nu,\alpha} (z,\bz|\xi)$ are connected to the holomorphic Hermite polynomials by
     \begin{align}\label{ExpHer}
      I_n^{\nu,\alpha} (z,\bz|\xi) =
       n! \sum_{k=0}^n \frac{\nu^k}{k!} \frac{(i)^{n-k} \alpha^{n-k/2}}{(n-k)!} H_{n-k}\left(i\alpha^{1/2} z + \frac{i \xi}{2 \alpha^{1/2}}\right) \bz^k
     \end{align}
\end{theorem}

\begin{proof}
By applying the binomial formula to \eqref{GhInPoly-1} and taking into account 
\eqref{partial-k}, we obtain
\begin{align*}
 I_n^{\nu,\alpha} (z,\bz|\xi)   & = (-1)^n  e^{-\alpha z^2-\xi z} \sum_{k=0}^n \binom{n}{k}  \left(\pzk( e^{\alpha z^2+\xi z}) \right) \left(-\nu \bz\right)^{n-k} \\
          & =  n! \sum_{k=0}^n  \frac{(i)^k \alpha^{k/2}}{k!} H_{k}\left(i\alpha^{1/2} z + \frac{i \xi}{2 \alpha^{1/2}}\right)
                               \frac{(\nu \bz)^{n-k}}{(n-k)!} .
\end{align*}
\end{proof}

\begin{remark}
	The $k$--th holomorphic component of $I_n^{\nu,\alpha}$ in \eqref{DerbarImG} is given by
	$$ h_k (z)  =  \frac{(i)^{n-k} \alpha^{n-k/2}}{(n-k)!n!} H_{n-k}\left(i\alpha^{1/2} z + \frac{i \xi}{2 \alpha^{1/2}}\right) .$$
\end{remark}

Added to the generalized Cauchy equation $\pbz^{n+1} =0$ satisfied by $I_n^{\nu,\alpha} (z,\bz|\xi)$, we can show that these  polynomials are common eigenfunctions of the partial differential operators
 of Laplacian type defined by
\begin{align}\label{Laplacianz}
\Delta^\nu_{\alpha,\xi} : = - \pzbz  + I_1^{\nu,\alpha} \pbz
\quad \mbox{and} \quad
\widetilde{\Delta}^\nu_{\alpha,\xi} : = - \pzbz   + I_1^{\nu,\alpha} \pz .
\end{align}

\begin{theorem}\label{thm-Eigen}
	The polynomials $I_n^{\nu,\alpha} (z,\bz|\xi)$ satisfy the partial differential equations
	\begin{align}\label{Eigen}
	\Delta^\nu_{\alpha,\xi}  I_n^{\nu,\alpha} (z,\bz|\xi) = \nu n I_n^{\nu,\alpha}
	\end{align}
	and
	\begin{align}\label{Eigenz}
	\widetilde{\Delta}^\nu_{\alpha,\xi}  I_n^{\nu,\alpha} (z,\bz|\xi) = - 2 \alpha n I_n^{\nu,\alpha}.
	\end{align}
\end{theorem}

\begin{proof} Using \eqref{DerbarIm} and  \eqref{Im3k}, we get
	\begin{align*}
	\Delta^\nu_{\alpha,\xi}  I_n^{\nu,\alpha} (z,\bz|\xi)&= \left(- \pz  + I_1^{\nu,\alpha}\right) \pbz I_n^{\nu,\alpha}
	\stackrel{\eqref{DerbarIm}}{=}\nu n \left(- \pz  + I_1^{\nu,\alpha}\right)  I_{n-1}^{\nu,\alpha}
	\stackrel{\eqref{Im3k}}{=}\nu n I_{n}^{\nu,\alpha}.
	\end{align*}
	The identity \eqref{Eigen} can be handled by applying $\pbz$ to the both sides of the recurrence relation \eqref{DerIm} involving $\pz$,
	and next use \eqref{DerbarIm} in Proposition \ref{prop-Repr2}.
	The identity \eqref{Eigenz} follows by proceeding in a similar way 
	using
	\eqref{DerImExp} and \eqref{Im3k}.
\end{proof}

\begin{remark}
According to the above result, the polynomials $I_n^{\nu,\alpha}$ are also eigenfunctions of
$$\Delta^\nu_{\alpha,\xi} \pm \widetilde{\Delta}^\nu_{\alpha,\xi}= (- \pz  + I_1^{\nu,\alpha}) ( \pz  \pm \pbz)$$
associated to the eigenvalue  $(\nu \mp 2\alpha) n$. In fact, the first order differential operators
$\pz  \pm \pbz$ are lowering operators for the polynomials $I_n^{\nu,\alpha}$ for satisfying
\begin{align}\label{Eigenz1bz}
(\pz  \pm \pbz)I_n^{\nu,\alpha} (z,\bz|\xi)= (\nu \mp 2\alpha) n I_{n-1}^{\nu,\alpha}
\end{align}
\end{remark}

\begin{remark}
	The polynomials $I_n^{\nu,\alpha}$ belong to the kernel of the operator $\nu \pz  + 2 \alpha \pbz$.
\end{remark}

 \section{Connection to rank--one automorphic functions}
In this section, we present an application in the context of the so--called automorphic functions of Landau type with respect to the $\mathbb{Z}$-character $\chi_\beta(k)=e^{2i\pi\beta k}$, i.e., the space of all complex--valued functions satisfying the functional equation
\begin{align}\label{FucEq}
f(z+k)=e^{2i\pi\beta k}  e^{2\alpha(z+\frac{k}{2})k} f(z)
\end{align}
for all $k\in \mathbb{Z}$ and $z\in \C$.
To this end, let $L^2(\C/\mathbb{Z},e^{-2\alpha |z|^2}dxdy)$ denote the space of $f: \C \longrightarrow \C$ satisfying \eqref{FucEq} and subject to the norm boundedness
on the strip $\C/\mathbb{Z}$ with respect to the gaussian measure
\begin{align}\label{normAutFct}
||f||^2_{\alpha,\mathbb{Z}}=\int_{\C/\mathbb{Z}}|f(z)|^2e^{-2\alpha |z|^2}dxdy <+\infty .
\end{align}
We denote by $\scal{ \cdot, \cdot }_{\alpha,\mathbb{Z}}$ the associated hermitian scalar product. Then it is proved in \cite{GhIn2013} that the functions
	\begin{align} \label{exAmn}
\varphi^{\nu,\alpha,\beta}_{m,n}(z,\bz)&
= (-i)^m  H_m^\alpha\left(  2\Im m(z)  + \frac{\pi(\beta+n)}{\alpha} \right) e^{\alpha,\beta}_n (z)
\end{align}
with $\Im m(z)=\frac{z-\overline{z}}{2i}$, form an orthogonal basis of $L^2(\C/\mathbb{Z},e^{-2\alpha |z|^2}dxdy)$.
This result can be reproved using the first order differential operator $A^*_{2\alpha}=-\partial_z+\nu \overline{z}$ and the corresponding  functions
$$\psi_{m,n}^{\nu,\alpha,\beta}(z,\bz) := (A^*_{2\alpha})^m( e^{\alpha,\beta}_n (z) ),$$
for $(m,n)\in \mathbb{Z}^+\times \mathbb{Z}$, where
\begin{align*}
e^{\alpha,\beta}_n (z) =e^{\alpha z^2+2i\pi (\beta+n)z}  
\end{align*}
In fact, we show that these functions form an orthogonal basis of $L^2(\C/\mathbb{Z},e^{-2\alpha |z|^2}dxdy)$ and that their explicit expression is given in terms of the special subclass
\begin{align}\label{GhImnPoly-1}
I_{m,n}^{\alpha,\beta}\left(z,\bz\right) & :=  I_{m}^{2\alpha,\alpha}\left(z,\bz|2i\pi(\beta+n)\right)
\\&= (-1)^n e^{\nu |z|^2 -\alpha  z^2 - 2i\pi(\beta+n) z}  \dfrac{\partial ^{n}}{\partial z^{n}}  \left(e^{- \nu |z|^2 + \alpha  z^2 + 2i\pi(\beta+n) z} \right),\nonumber
\end{align}
where $\alpha>$, $\beta\in\R$, $m=0,1,2,\cdots,$ and $n\in\Z^{+}$.

\begin{lemma}\label{lemexAmn}
	We have
			\begin{align*}
	\psi_{m,n}^{\nu,\alpha,\beta}(z,\bz) =
	 I^{\frac{\nu}{2},\nu}_m(z,\bz|2i\pi(\beta+n)) e^{\alpha,\nu}_n =\varphi^{\nu,\alpha,\beta}_{m,n}(z,\bz).
		\end{align*}
\end{lemma}

\begin{proof}
	Since $A^*_{2\alpha}f
	= - e^{2\alpha|z|^2} \partial_z \left( e^{-2\alpha|z|^2} f \right) $, we get $(A^*)^m_{2\alpha}f= (-1)^m e^{2\alpha|z|^2} \partial_z^m \left( e^{2\alpha|z|^2} f \right) $. Therefore,
	\begin{align*}
(A^*_{2\alpha})^m( e^{\alpha,\beta}_n (z) )
&	= (-1)^m e^{2\alpha|z|^2} \partial_z^m \left( e^{-2\alpha|z|^2} e^{\alpha,\beta}_n (z) \right)
	=  I_{m,n}^{2\alpha,\alpha,\beta}\left(z,\bz\right)  e^{\alpha,\beta}_n (z)
	\end{align*}
\end{proof}

Such claim (i.e., $\psi_{m,n}^{\nu,\alpha,\beta}(z,\bz)$ is an orthogonal basis of $L^2(\C/\mathbb{Z};e^{-2\alpha |z|^2}dxdy)$) is based on the following

\begin{lemma}[\cite{GhIn2013}] \label{enBasisF}
	The functions $ e^{\alpha,\beta}_n (z) =e^{\alpha z^2+2i\pi (\beta+n)z}$; $n=0,1,\cdots$, form a complete system of the theta Bargmann-Fock space
	$\mathcal{F}^{2,2\alpha}_{\mathbb{Z},\beta}(\C)$ of all complex-valued entire functions satisfying \eqref{FucEq}) and belonging to $L^2(\C/\mathbb{Z},e^{-2\alpha |z|^2}dxdy)$.
\end{lemma}

\begin{lemma} The functions $\psi_{m,n}^{\nu,\alpha,\beta}$ are eigenfunctions of $\Delta_{2\alpha}$ associated to the eigenvalue $2\alpha m$.
\end{lemma}

\begin{proof}
	The result readily follows by induction. It is clear for $m=0$.
	Next, if $ \Delta_{2\alpha}\psi_{k,n}^{\nu,\alpha,\beta}=2\alpha k \psi_{k,n}^{\nu,\alpha,\beta}$ is verified for $k\leq m$, we use the fact $\Delta_\nu = A^*_{2\alpha}A = AA^*_{2\alpha}$ to get
	\begin{align*}
	\Delta_{2\alpha}\psi_{m+1,n}^{\nu,\alpha,\beta}
	&=(A^*A)A^*(A^*)^{m}( e^{\alpha,\nu}_n)
	=A^*(\nu+\Delta_\nu)(A^*)^{m}( e^{\alpha,\nu}_n)
	=\nu(m+1)\psi_{m+1,n}^{\nu,\alpha,\beta}
	\end{align*}
\end{proof}

\begin{lemma}
	The functions $\psi_{m,n}^{\nu,\alpha,\beta}$ satisfy the functional equation \eqref{FucEq} and are orthogonal in $L^2(\C/\mathbb{Z},e^{-2\alpha |z|^2}dxdy)$.	
\end{lemma}

\begin{proof}
	Notice that for $f=e^{\alpha z^2+2i\pi \beta z}  F$ and $g= e^{\alpha z^2+2i\pi \beta z}  G$ satisfying the autoumorphic equation \eqref{FucEq}, the functions $F$ and $G$ are $\Z$--periodic and we have
	$$
	\scal{ f,g }_{\alpha,\mathbb{Z}} = \int_{[0,1]\times \R} |g(z)|^2e^{-4\alpha y^2-4\pi\beta y}dxdy . $$
	
	Then, the result follows by a tedious but straightforward computations using the explicit expression of $\psi_{m,n}^{\nu,\alpha,\beta}$ given by Lemma \ref{lemexAmn} combined with the orthogonality of $e^{\alpha,\beta}_n $ (see Lemma \ref{enBasisF}).
\end{proof}

\begin{remark}
	The above discussion shows that the polynomials $I_{m,n}^{\alpha,\beta}\left(z,\bz\right)$ in \eqref{GhImnPoly-1} characterize the orthogonal complement of $\mathcal{F}^{2,2\alpha}_{\mathbb{Z},\beta}(\C)$ in the full Hilbert space $L^2(\C/\mathbb{Z};e^{-2\alpha |z|^2}dxdy)$.
\end{remark}

\section{Concluding remarks}
In the present paper, we have discussed in Sections 2, 3, 4, 5 and 6 some basic properties of a novel class of polyanalyic polynomials of Hermite type such as operational and integral representations, generating functions, orthogonality relations and different differential equations they satisfy. In Section 6, we have proved that $I_{n}^{\nu,\alpha}\left(z,\bz|\xi\right)$ are  eigenfunctions of the partial differential operator $\Delta^\nu_{\alpha,\xi} : = - \pzbz  + I_1^{\nu,\alpha} \pbz$ in \eqref{Laplacianz}. It should be mentioned here that the particular case  $\Delta^\nu_{0,0}$ is the twisted Laplacian $\Delta_\nu =  - \PbP + \nu  \bz \frac{\partial}{\partial {\bz}}$ describing in physics the quantum behavior of a nonrelativistic  charged particle confined in the plane under the action of an external constant magnetic field \cite{Shigekawa87,Folland89,Thangavelu98,FerapontovVeselov01}.
The corresponding $L^2$--spectral analysis on the free Hilbert space $L^2(\C; e^{-\nu|z|^2} dxdy)$ is completely described by the univariate complex Hermite polynomials $H_{m,n}^\nu(z,\bz )$ that form an orthogonal basis of $L^2$--eigenfunctions of  $\Delta_\nu$ (see \cite{Ito52,IntInt06,Gh08JMAA}).
While the subclass $I_{m,n}^{\nu,\alpha}\left(z,\bz|2i\pi(\beta+n)\right)$ appeared to be essential in describing the spectral theory of $\Delta_\nu$ (with $\nu=2\alpha$) acting on rank--one automorphic functions belonging to $L^2(\C/\Z; e^{-\nu|z|^2} dxdy)$ .

Accordingly, we claim that the polynomials  $I_{n}^{\nu,\alpha}\left(z,\bz|\xi\right)$ will play a crucial rule in illustrating the spectral properties of $\Delta^\nu_{\alpha,\xi}$ acting on an appropriate Hilbert space whose associated scaler product is predicted 
by Theorem \ref{thm-orthogIm2}. The corresponding spaces will constitute the polyanalytic version of the functional spaces studied in \cite{Eijndhoven-Meyers1990,Chihara2017}. This will be considered in detail in a forthcoming paper.



\begin{thebibliography}{99}
	
\bibitem{AbreuFeichtinger2014} Abreu L.D., Feichtinger H.G.,
	Function spaces of polyanalytic functions.
	Harmonic and complex analysis and its applications, 1--38, Trends Math., Birkhäuser/Springer, Cham, 2014.
\bibitem{ABEG2015} Agorram F., Benkhadra A., El Hamyani A.,  Ghanmi A.,
	{\it Complex Hermite functions as Fourier-Wigner transform}.
	Integral Transforms Spec. Funct. 27 (2016), no. 2, 94--100.
\bibitem{Balk1991} Balk M.B.,
	Polyanalytic functions,
	Mathematical Research, vol. 63, Akademie-Verlag, Berlin, 1991.
\bibitem{BenahmadiG2018} Benahmadi A.,  Ghanmi A.,
	Non-trivial 1d and 2d integral transforms of Segal-Bargmann type.
	To appear in Integral transforms and special functions.
\bibitem{Burgatti1922} Burgatti P.,
	Sulla funzioni analitiche d'ordini $n$.
	Boll. Unione Mat. Ital. 1 (1922) 1, 8--12.
\bibitem{Chihara2017} Chihara H.,
	Holomorphic Hermite functions and ellipses.
	Integral Transforms Spec. Funct. 28 (2017), no. 8, 605--615.
\bibitem{CoftasGazeau10} Cotfas N.,  Gazeau J-P., Górska K.,
	{\it Complex and real Hermite polynomials and related quantizations},
	{ J. Phys. A}  43, no. 30 (2010) 305304, 14 pp.
\bibitem{deGosson2017} de Gosson M.,
   The Wigner transform. Advanced Textbooks in Mathematics. World Scientific Publishing Co. Pte. Ltd., Hackensack, NJ, 2017.
\bibitem{AiadGh10JMAA}  El Gourari A., Ghanmi A.,
	{\it  Spectral analysis on planar mixed automorphic forms}.
	J. Math. Anal. Appl. 383 (2011), no. 2, 474--481.
\bibitem{FerapontovVeselov01}
	Ferapontov E.V., Veselov A.P.,
	Integrable Schrödinger operators with magnetic fields: factorization method on curved surfaces.
	{\it J. Math. Phys.} (2) \textbf{42} (2001)  590--607.
\bibitem{Folland89}  Folland G B.,
	{\it Harmonic analysis in phase space}.
	Princeton university press, New Jersey, 1989.
\bibitem{GazeauSzafraniec2011}  Gazeau J.P., Szafraniec F.H.,
	Holomorphic Hermite polynomials and a non-commutative plane,
	J. Phys. A 44 (2011), no. 49, 495201, 13.
\bibitem{Gh08JMAA}  Ghanmi A.,
	{\it A class of generalized complex Hermite polynomials}.
	{ J. Math. Anal. App.}  340 (2008), 1395-1406.
\bibitem{Gh13ITSF}  Ghanmi A.,
	{\it Operational formulae for the complex Hermite polynomials $H_{p,q}(z, \bz)$}.
	{Integral Transforms Spec. Funct.},  Volume 24, Issue 11 (2013) pp  884-895.
\bibitem{Gh2017} Ghanmi A.,
	Mehler's formulas for the univariate complex Hermite polynomials and applications. Math. Methods Appl. Sci. 40 (2017), no. 18, 7540--7545.
\bibitem{GhIn2013} Ghanmi A., IntissarA.,
     Construction of concrete orthonormal basis for $L^2-(\Gamma,\chi)$--theta functions
	associated to discrete subgroups of rank--one in $(\c,+)$.
	J. Math. Phys. 54 (6):063514, 2013.
\bibitem{Hermite1864-1908}   Hermite C.,
	{\it Sur un nouveau développement en série des fonctions},
	Compt. Rend. Acad. Sci. Paris 58, t. LVIII (1864) 94-100 et 266-273
	ou Oeuvres complètes, tome 2. Paris, p. 293-308, 1908.
\bibitem{IntInt06} Intissar A., Intissar A.,
	{\it Spectral properties of the Cauchy transform on $L^2(\C;e^{-|z|^2}d\lambda)$},
	{ J. Math. Anal. Appl.} 313, no 2 (2006) 400-418.
\bibitem{Ismail13a} Ismail M.E.H., Simeonov P.,
	{\it Complex Hermite polynomials: their combinatorics and integral operators}.
	To appear in Proceeding of the AMS (2014).
\bibitem{IsmailTrans2016} Ismail M.E.H.,
	{\it Analytic properties of complex Hermite polynomials}.
	Trans. Amer. Math. Soc. 368 (2016), no. 2, 1189-1210.
\bibitem{Ito52}  It\^o K.,
	{\it Complex multiple Wiener integral}.
	{ Jap. J. Math.}, 22 (1952) 63-86
\bibitem{Karp2001}  Karp D.,
	Holomorphic spaces related to orthogonal polynomials and analytic continuation of functions, Analytic extension formulas and their applications
	(Fukuoka, 1999/Kyoto, 2000), Int. Soc. Anal. Appl. Comput., vol. 9, Kluwer Acad. Publ., Dordrecht, 2001, pp. 169--187.
\bibitem{Matsumoto96} Matsumoto, H.,
     Quadratic Hamiltonians and associated orthogonal polynomials,
     {\it J. Funct. Anal.} 136 (1996), no. 1, 214--225.
\bibitem{Mehler1866} Mehler F.G.
	{Ueber die Entwicklung einer Function von beliebig vielen Variabeln nach Laplaceschen Functionen h\"oherer Ordnung}.
	{\it J. Reine Angew. Math.} 1866; 66:161--176.
\bibitem{Rainville71}  Rainville E.D.,
	{\it Special functions}, Chelsea Publishing Co., Bronx, N.Y., 1971.	
\bibitem{Shigekawa87}   Shigekawa I.,
	Eigenvalue problems of Schrödinger operator with magnetic field on compact Riemannian manifold,
	{\it J. Funct. Anal.}, 75 (1987) 92-127.
\bibitem{Souid2015} Souid El Ainin M.,
       Concrete description of the $(\Gamma,\chi)$--theta Fock-Bargmann space for rank--one in high dimension.
       Complex Var. Elliptic Equ. 60 (2015), no. 12, 1739--1751.
\bibitem{Szego75}  Szeg\"o G.,
	{\it Orthogonal polynomials. Fourth edition},
	American Mathematical Society, Providence, R.I., 1975.
\bibitem{Thangavelu93}  Thangavelu S.,
	{\it Lectures on Hermite and Laguerre Expansions}.
	Princeton University Press, 1993.
\bibitem{Thangavelu98} Thangavelu S.,
	{\it Harmonic analysis on the Heisenberg group}.
	Progress in Mathematics, 159. Birkh\"auser Boston, Inc., Boston, MA, 1998.
\bibitem{Eijndhoven-Meyers1990} van Eijndhoven S.J.L.,  Meyers J.L.H.,
	{\it New orthogonality relations for the Hermite polynomials and related Hilbert spaces}.
	J. Math. Anal. Appl. 146 (1990), no. 1, 89--98.
\bibitem{Wong98}  Wong M.W.,
	Weyl Transforms, Springer-Verlag, 1998.
\end{thebibliography}
\end{document}